\documentclass[12pt,a4paper]{amsart}
\usepackage{latexsym}
\usepackage{amsmath,amssymb,exercise}
\usepackage{dsfont}

\usepackage[bookmarks=true,hyperindex,pdftex,colorlinks,citecolor=blue]{hyperref}

\newtheorem{theo}{Theorem}[section]
 \newtheorem{definition}[theo]{Definition}
 
 \newtheorem{lem}[theo]{Lemma}
\newtheorem{propos}[theo]{Proposition}

 \newtheorem{corollary}[theo]{Corollary}

\newcommand{\R}{{\mathbb{R}}}

\newcommand{\N}{{\mathbb{N}}}

\newcommand{\eps}{\varepsilon}

\newcommand{\loglike}[1]{\mathop{\rm #1}\nolimits}
\newcommand{\ex}{\loglike{ext}}

\newcommand{\dist}{\rho}

\newcommand{\bea}{\begin{eqnarray*}}
\newcommand{\eea}{\end{eqnarray*}}
\newcommand{\beq}{\begin{eqnarray}}
\newcommand{\eeq}{\end{eqnarray}}

\newcommand{\vol}{\operatorname{vol}}
\newcommand{\Face}{\operatorname{\mathfrak{Face}}}
\newcommand{\Slice}{\operatorname{\mathfrak{Slice}}}
\newcommand{\ifff}{if and only if }

 \renewcommand{\le}{\leqslant}
\renewcommand{\leq}{\leqslant}
\renewcommand{\ge}{\geqslant}
\renewcommand{\geq}{\geqslant}

\numberwithin{equation}{section}

\begin{document}

\title[Generalized-lush spaces]{Generalized-lush spaces revisited}

\address{ School of Mathematics and Informatics, V.N. Karazin Kharkiv National University, 61022 Kharkiv, Ukraine}
\author[Kadets]{V. Kadets}

\address[Kadets]{ \href{http://orcid.org/0000-0002-5606-2679}{ORCID: \texttt{0000-0002-5606-2679}}}
\email{v.kateds@karazin.ua }

\author[Zavarzina]{O. Zavarzina}
\address[Zavarzina]{\href{http://orcid.org/0000-0002-5731-6343}{ORCID: \texttt{0000-0002-5731-6343}}}
\email{olesia.zavarzina@yahoo.com}


\thanks{The research is done in frames of Ukrainian Ministry of Science and Education Research Program 0118U002036}

\begin{abstract}
We study geometric properties of GL-spaces. We demonstrate that every finite-dimensional GL-space is polyhedral; that in dimension 2 there are only two, up to isometry, GL-spaces, namely the space whose unit sphere is a square (like $\ell_\infty^2$ or $\ell_1^2$) and the space whose unit sphere is an equilateral hexagon. Finally, we address the question what are the spaces $E = (\R^n, \|\cdot\|_E)$ with absolute norm such that for every collection $X_1, \ldots, X_n$ of GL-spaces their $E$-sum is a GL-space.
\end{abstract}

\maketitle

\section{Introduction}

 In 1987, Daryl~Tingley \cite{ting} posed the following question: let $f$ be a bijective isometry between the unit spheres $S_X$ and $S_E$ of real Banach spaces $X$, $E$ respectively. Is it true that $f$ extends to a linear isometry $F: X \to E$ of the corresponding spaces?

Considering Tigley's question Lixin Cheng and Yunbai Dong \cite{cheng} introduced the following useful terminology:

\begin{definition} \emph{A Banach space $X$ is said to have the} Mazur--Ulam
property \emph{provided that for every Banach space $E$, every surjective isometry $f \colon S_X \to S_E$ is the restriction of a linear isometry between $X$ and $E$}.
\end{definition}

There is a number of publications devoted to Tingley's problem
(say, Zentralblatt Math. shows 57 related papers published from 2002 to 2019) and, in particular, it is known \cite{kad-mar} that finite-dimensional polyhedral spaces (i.e. those spaces whose unit ball is a polyhedron) have the Mazur--Ulam property. Surprisingly, for general spaces the innocently-looking Tigley's question remains unanswered even in dimension two.

 In 2013 Dongni Tan, Xujian Huang, and Rui Liu \cite{TanHuangLiu} have made a substantial advance. In order to explain it let us give some definitions.

 The distance from a point $x$ of a normed space $X$ to a non-empty subset $A\subset X$ is defined to be the infimum of the distances from $x$ to the elements of~$A$: $\dist (x,A) = \mathop{\inf}_{a \in A} \|x - a\|$.
\begin{definition}
\emph{A} closed slice of the unit ball $B_X$ \emph{of a Banach space} $X$ \emph{is a subset of }$B_X$ \emph{of the form}
$$
\Slice(x^*,\alpha)=\{ x\in B_X: x^*(x) \geq 1-\alpha\},
$$
\emph{where $x^*\in S_{X^*}$ and $\alpha\in(0,1)$}.
\end{definition}
\begin{definition}[{\cite{TanHuangLiu}}]
\emph{A Banach space $X$ is said to be} generalized-lush \emph{(GL-space for short) if for every $x\in S_X$ and every $\eps > 0$ there exists a slice $\Slice(x^*,\eps)$ with $x^*\in S_{X^*}$ such that $x\in \Slice(x^*,\eps)$ and
$$\dist(y,\Slice(x^*,\eps))+\dist(- y, \Slice(x^*,\eps)) < 2+\eps$$
for all $y\in S_X$. A Banach space $E$ is said to be a} local-GL-space \emph{if for
every separable subspace $Y \subset E$ there is a GL-subspace $X \subset E$ such that $Y \subset X \subset E$.}
\end{definition}

These definitions are motivated by the concept of lush spaces, introduced in \cite{BKMW} in relation to the numerical index of operators, see Section 1.5 in the recent monograph \cite{Spears} for the definition, examples and the history of the subject.

In \cite{TanHuangLiu} it is demonstrated that all local-GL-spaces (and consequently all GL-spaces, all lush spaces and in particular all $C(K)$ and $L_1(\mu)$ spaces) enjoy the Mazur--Ulam property. Also it is shown that the class of GL-spaces is stable under $c_0$-, $\ell_1$- and $\ell_{\infty}$-sums and that the space $C(K,X)$ is a GL-space whenever $X$ is a GL-space, which gives a number of examples of spaces with the Mazur--Ulam property. In the same vein Jan-David Hardtke \cite{Hardtke} demonstrated that the class of GL-spaces is stable under ultraproducts and under passing to a large class of F-ideals, in particular to M-ideals.

In this article we, at first, demonstrate that every finite-dimensional GL-space is polyhedral (Theorem \ref{theo-main1}), that is in the finite-dimensional case the results of \cite{TanHuangLiu} do not give new examples of Mazur--Ulam spaces comparing to \cite{kad-mar}. At second, in Theorem \ref{theo-main2} we show that in dimension 2 there are only two, up to isometry, GL-spaces, namely the space whose unit sphere is a square (like $\ell_\infty^2$ or $\ell_1^2$) and the space whose unit sphere is an equilateral hexagon. The above results are collected in Section \ref{sec-polyhedr}. Finally, in Section \ref{sec-sums} we address the question what are the spaces $E = (\R^n, \|\cdot\|_E)$ with absolute norm such that for every collection $X_1, \ldots, X_n$ of GL-spaces their $E$-sum is a GL-space.

In the exposition we will use the standard Banach space theory notations. In particular, for Banach space $X$ we denote, as we already have done above, by $B_X$, $S_X$ and $X^*$ the closed unit ball, unit sphere and the dual space respectively. In our paper we consider only real Banach spaces and real linear functionals (otherwise the definition of slice should be modified). All unexplained functional analysis terminology can be found in \cite{Kad}.

\section{Polyhedrality of finite-dimensional GL-spaces} \label{sec-polyhedr}

Recall that the \emph{Hausdorff distance} between two not empty closed subsets $A$ and $B$ of a metric space $X$ is defined as
$$
\rho_{\rm H} (A,B) = \max \left\{{\mathop {\sup }\limits_{b \in B} \rho (b,A),\;\mathop
{\sup }\limits_{a \in A} \rho (a,B)} \right\}.
$$
According to the Blaschke selection theorem (see \cite[Theorem 2.5.14]{Thompson}) the collection of not empty closed convex subsets of a given bounded subset of a finite-dimensional normed space forms a compact in the Hausdorff metric.
\begin{definition}
\emph{A} face of the unit ball \emph{of a Banach space} $X$ \emph{is a not-empty set of the form}
$$
\Face(x^*) = \{ x\in B_X: x^*(x) = 1\},
$$
\emph{where} $x^*\in S_{X^*}$.
\end{definition}
\begin{definition} \label{defin-plump}
\emph{We call a subset} $A \subset B_X$ plump \emph{if for every $y\in S_X$  there are $a_1, a_2 \in A$ such that}
$$
\|y - a_1\| + \|y + a_2\| \leq 2.
$$
\end{definition}
Remark, that in the case of compact $A$ the distances attain, so the inequality  $\dist(y, A)+\dist(- y, A) \leq 2$  implies automatically the existence of corresponding $a_1, a_2 \in A$; that if $A \subset B \subset B_X$ and $A$ is plump then $B$ is plump as well; and if for every $\eps > 0$ every point $x \in S_X$ is contained in a plump slice $\Slice(x^*,\eps)$ with $x^*\in S_{X^*}$ then $X$ is a GL-space.

The following definition is motivated by \cite[Proposition 2.2]{Hardtke} and by analogous concepts of ultra-lush spaces from \cite{BKMM2009} and rigid narrow operator with respect to a subset from \cite{BKSW2005}.

\begin{definition} \label{def-U-GL} \emph{A normed space $X$ is said to be}  ultra-GL with respect to a subspace $W \subset X^*$ \emph{(ultra-GL($W$)-space) if for every $x\in S_X$ there exists an $x^*\in S_{W}$ such that $x \in \Face(x^*)$ and $\Face(x^*)$ is plump. In the case of $W = X^*$ the space $X$ is said to be} ultra-GL.
\end{definition}

\begin{lem}
Let $X$ be a finite-dimensional normed space. Then the following conditions are equivalent:
\begin{enumerate}
 \item $X$ is a GL-space;
 \item $X$ is ultra-GL.
 \item for every $x\in S_X$ there exists a convex plump subset $B\subset S_X$ such that $x\in B$.
\end{enumerate}
\end{lem}
\begin{proof}
$\mathbf{(3) \Rightarrow (2)}$ Let the third condition hold. Then for given $x\in S_X$ there is a convex plump subset $B\subset S_X$ containing $x$. This $B$ can be separated from the open unit ball by a norm-one functional (Hahn-Banach), so there exists a functional $x^*\in S_{X^*}$ such that $B \subset \Face(x^*)$. Then $x \in \Face(x^*)$ and $\Face(x^*)$ is plump.

$\mathbf{(2) \Rightarrow (1)}$ Take for given $x\in S_X$ the corresponding $x^*\in S_{X^*}$ that generates a plump face containing $x$. For every $\eps>0$ consider the slice $\Slice(x^*,\eps)$. It is obvious that $\Face(x^*) \subset \Slice(x^*,\eps) \subset B_X$, so $\Slice(x^*,\eps)$ is plump.

$\mathbf{(1)\Rightarrow(3)}$ Observe, that (1) means for every $n\in \N$ there exists $x_n^*\in S_{X^*}$ such that $x\in \Slice(x_n^*,\frac{1}{n})=S_n$ and
\begin{align}\label{eq*}
 \dist(y, S_n)+\dist(y, - S_n) < 2+\frac{1}{n}
\end{align}
for every $y\in S_X$.
The Blaschke selection theorem implies the existence of a subsequence $S_{n_k}$ that converges in the Hausdorff metrics to a closed convex set
$$
B =\lim_{k\to \infty}S_{n_k} \subset B_X.
$$
Obviously, the inclusion $S_n \subset \{x \in B_X \colon \|x\| \ge 1 - \frac{1}{n}\}$ implies that the limiting set $B$ lies on the unit sphere. Since for a fixed $y\in S_X$ the distance $\dist(y, S)$ depends continuously in the Hausdorff metrics on the variable $S$, \eqref{eq*} gives us the desired inequality $\dist(y,B)+\dist(y,-B)\leq 2$.
\end{proof}

Note the following obvious corollary.

\begin{corollary}\label{cor-plump-union}
Let $X$ be a finite-dimensional GL-space. Then $S_X$ is the union of its plump faces.
\end{corollary}

\begin{lem} \label{lemma-plump-face-description}
Let $X$ be a normed space, $x^*\in S_{X^*}$. Then the following conditions are equivalent:
\begin{enumerate}
 \item $\Face(x^*)$ is plump;
 \item $\dist(y, \Face(x^*))+\dist(- y, \Face(x^*)) \leq 2$ for every $y \in B_X$ and the distances are attained.
 \item $\dist(y, \Face(x^*)) = 1 - x^*(y)$ for every $y \in B_X$ and the distance is attained.
 \item $\dist(y, \Face(x^*)) = 1 - x^*(y)$ for every $y \in S_X$ and the distance is attained.
\end{enumerate}
\end{lem}
\begin{proof} The implications $(2)\Rightarrow(1)$ and $(3)\Rightarrow(4)$ are obvious, so it remains to prove the implications $(1)\Rightarrow(2)$, $(2)\Rightarrow(3)$ and $(4)\Rightarrow(1)$.

\vspace{2 mm}
{$\mathbf{(1)\Rightarrow(2)}$}. For a given $y \in B_X$ denote $\hat y = \frac{y}{\|y\|}$. Using  plumpness of $\Face(x^*)$ let us choose $u, v \in \Face(x^*)$ such that $\|u - \hat y\| + \|v + \hat y\| \le 2$. Denote $\lambda = \|y\|$. Then we have
\begin{align*}
\|u - y\| &+ \|v + y\| = \|\lambda(u - \hat y) + (1 - \lambda)u\| + \|\lambda(v + \hat y) + (1 - \lambda)v\| \\
& \le \lambda\|(u - \hat y)\| + (1 - \lambda)\|u\| + \lambda\|(v + \hat y)\| + (1 - \lambda)\|v\|
\\
&= \lambda\left(\|(u - \hat y)\| + \|(v + \hat y)\|\right) + (1 - \lambda) + (1 - \lambda) \le 2.
\end{align*}

{$\mathbf{(2)\Rightarrow(3)}$}. Now for a given $y \in B_X$ select $u, v \in \Face(x^*)$ with $\|u - y\| = \dist(y, \Face(x^*))$ and $\|v + y\| = \dist(-y, \Face(x^*))$. Then $\|u - y\| + \|v + y\| \le 2$, $\|u - y\| \ge x^*(u-y) = 1 - x^*(y)$ and $\|v + y\| \ge x^*(v+y) = 1 + x^*(y)$. Combining these three inequalities we obtain that
$$
 2 \geq \|u - y\| + \|v + y\| \geq 1 - x^*(y) +1 + x^*(y) = 2.
$$
This may happen only if $\|u - y\| = 1 - x^*(y)$ and $\|v + y\| = 1 + x^*(y)$.

\vspace{2 mm}
$\mathbf{(4)\Rightarrow(1)}$. Let $y \in S_X$, then $-y \in S_X$ so the condition (4) gives us $\dist(y, \Face(x^*)) = 1 - x^*(y)$ and $\dist(-y, \Face(x^*)) = 1 + x^*(y)$. Consequently,
 $\dist(y, \Face(x^*)) + \dist(-y, \Face(x^*)) = 2$.
\end{proof}

For a subset $A \subset X$ define its difference body $A - A$ as the Minkowski sum of $A$ and $-A$:
$$
A - A = \{x - y \colon x, y \in A\}.
$$

\begin{lem}\label{lem-plump-face}
Let $x^*\in S_{X^*}$ and $\Face(x^*)$ be plump. Then
$$
\Face(x^*) - \Face(x^*) \supset B_X \cap \ker x^*.
$$
\end{lem}
\begin{proof}
Fix a $y \in B_X \cap \ker x^*$. Our goal is to show that $y \in \Face(x^*) - \Face(x^*)$. Indeed, select $v \in \Face(x^*)$ with $\|v + y\| = \dist(-y, \Face(x^*))$. According to item (3) of Lemma \ref{lemma-plump-face-description} $\|v + y\| = x^*(v + y) = 1$. Consequently, $v + y \in \Face(x^*)$, and $y = (v + y) - v \in \Face(x^*) - \Face(x^*)$. \end{proof}

\begin{theo}\label{theo-main1}
The unit ball of every finite-dimensional GL-space is a polyhedron whose faces are plump.
\end{theo}
\begin{proof}
Let $X=(\R^n,\|\cdot\|)$ be a GL-space, and let $r >0$ be the minimal $(n-1)$-dimensional volume of intersections of $B_X$ with $(n-1)$-dimensional linear subspaces of $X$. Fix a plump face $\Face(x^*)$ and an $x \in \Face(x^*)$. Then $\Face(x^*) - x \subset \ker x^*$. Remark that due to the previous Lemma
$$
(\Face(x^*) - x) - (\Face(x^*) - x) = \Face(x^*) - \Face(x^*) \supset B_X \cap \ker x^*.
$$
According to the Rogers--Shephard theorem \cite[Theorem 1]{RogersSheph}, for every convex body $K$ in an $m$-dimensional space $E$
$$
\vol_{m}(K - K) \leq {{2m}\choose{m}} \vol_{m}(K).
$$
Applying that theorem to $K = \Face(x^*) - x \subset E := \ker x^*$ we obtain that
\begin{align*}
\vol_{n-1}(\Face(x^*)) &= \vol_{n-1}(\Face(x^*) - x) \\
&\geq {{2m}\choose{m}}^{-1} \vol_{n-1}\left((\Face(x^*) - x) - (\Face(x^*) - x)\right) \\
& \geq {{2m}\choose{m}}^{-1} \vol_{n-1}\left(B_X \cap \ker x^*\right) \geq {{2m}\choose{m}}^{-1} r.
\end{align*}
So the $(n-1)$-dimensional volumes of plump faces are bounded below by a positive number, so the set of plump faces of $B_X$ is finite. Enumerate $\{\Face(x_k^*)\}_{k=1}^N$, $x_k^* \in S_{X^*}$, all plump faces. Our Corollary \ref{cor-plump-union} ensures that $S_X = \bigcup_{k=1}^N \Face(x_k^*)$, which means that $B_X = \{ x \in \R^n \colon |x_k^*(x)| \le 1 \textrm{ for all } k = 1, 2, \ldots, N\}$.
\end{proof}

Before moving to our second main theorem we will set out the result, needed for the proof.
\begin{propos}[{\cite[Proposition 20]{MarSwanW}}]\label{Uball}
 On the unit sphere of a Minkow\-ski plane there are at most three
pairs of segments of length at least 1. If there are three pairs of segments of length
at least 1, then the unit disc must be a hexagon with vertices $\pm x_1,\pm x_2,\pm \lambda(x_1+x_2)$
for some $\lambda \in (\frac{1}{2},1]$, and at least two pairs are of length exactly 1.
\end{propos}

It is known  \cite{TanHuangLiu} that two-dimensional spaces $\R^2$ with the norm
$$
\|(a_1, a_2)\|_\infty = \max\{|a_1|, |a_2|\} \textrm{ or } \|(a_1, a_2)\|_1 = |a_1| + |a_2|
$$
(these space are called $\ell_\infty^2$ and $\ell_1^2$ respectively), and $\R^2$ equipped with the norm $\|(a_1, a_2)\| = \max\{|a_2|, |a_1| + \frac{1}{2}|a_2|\}$ (let us denote it $\tilde E$) are GL-spaces. The first two spaces are mutually isometric (their unit spheres are parallelograms), so as Banach spaces they are indistinguishable, and the unit sphere of $\tilde E$ is an equilateral in the metric of $\tilde E$ hexagon. Our next theorem says that there are no other (up to isometry)  two-dimensional  GL-spaces.
\begin{theo}\label{theo-main2}
Let $X$ be a real two-dimensional GL-space. Then $X$ is isomertic either to two-dimensional space $\R^2$ with such a norm that the unit ball is a parallelogram, or to the space $\tilde E$.
\end{theo}
\begin{proof}
 According to Theorem \ref{theo-main1}, the unit ball of $X$ is a polygon whose edges are plump. Lemma \ref{lem-plump-face} implies that the length of each edge is not smaller than 1.
Application of Proposition \ref{Uball} implies, that the unit ball of the space $X$ either is a parallelogram (and then the space is isometric to two-dimensional space $\ell_{\infty}^2$), or a hexagon with vertices $\pm x_1,\pm x_2,\pm \lambda(x_1+x_2)$
for some $\lambda \in (\frac{1}{2},1]$.  It remains to consider the case of hexagon.
Let us choose the basis vectors $e_1=x_1, e_2=x_2$. The coordinates of the vertices in this basis are: $(\pm1,0),(0,\pm1)$ and $(\pm\lambda,\pm\lambda)$. So $X$ may be considered as $\R^2$ equipped with the norm $\|(u, v)\| = \|u e_1 + v e_2\|$. Now we may explicitly write down an expression for the norm in this space. For every $x=(u,v)\in(\R^2,\|\cdot\|)$
$$
\|x\|= \max_{1\leq i\leq3}\{|f_i(x)|\},
$$
where $f_i(x), 1\leq i \leq 3$ are linear functionals of the form: $f_1(x)=u+\frac{1-\lambda}{\lambda}v$, $f_2(x)=\frac{1-\lambda}{\lambda}u+v$ and $f_3(x)=-u+v$.
Now, let us consider the edge $[\lambda(x_1+x_2), x_2] = \Face(f_2)$ and compute
$\dist(x_1,\Face(f_2))$.
\bea
\dist(x_1, \Face(f_2))&=&\inf_{\alpha\in[0,1]}\dist(x_1,\alpha\lambda x_1+\alpha\lambda x_2+(1-\alpha)x_2)\\
&=&\inf_{\alpha\in[0,1]}\|(1-\alpha\lambda)x_1+(\alpha-1-\alpha\lambda)x_2\|\\
&=&\inf_{\alpha\in[0,1]}\max_{1\leq i\leq3}\{|f_i(1-\alpha\lambda,\alpha-1-\alpha\lambda)|\}\\
&=&\inf_{\alpha\in[0,1]}f_3(1-\alpha\lambda,\alpha-1-\alpha\lambda)\\
&=&\inf_{\alpha\in[0,1]}(2-\alpha)=1.
\eea
On the other hand, since $\Face(f_2)$ is plump, $\dist(x_1,\Face(f_2)) = 1 - f_2(x_1) = 1 - \frac{1-\lambda}{\lambda}$, so $\lambda=1$. Summarizing, in the hexagonal case our space is isometric to $\R^2$ endowed with the norm
$$
\|x\|=\|(u,v)\|= \max\{|u|, |v|, |u - v|\}.
$$
This space is isometric to $\tilde E$.
\end{proof}

The following property of  $\tilde E$ will be helpful in the next section.

\begin{propos}\label{propA}
Let $\tilde E = \R^2$ equipped with the norm $\|(a_1, a_2)\| = \max\{|a_2|, |a_1| + \frac{1}{2}|a_2|\}$, $\tilde e_1 = (1, 0)$, $\tilde e_2 = (\frac{1}{2}, 1)$,  $\tilde e_3 = (- \frac{1}{2}, 1)$,  $\tilde e_2^* \in S_{\tilde E^*}$ be the second coordinate functional on $\R^2$:  $\tilde e_2^*(a, b) = b$. Then $\Face(\tilde e_2^*)$ is equal to the  edge of the hexagon $B_{\tilde E}$ that connects $\tilde e_2$ and $\tilde e_3$, and the vertex $\tilde e_1$ belongs to the kernel of $\tilde e_2^*$. Let $t \in [0, 1]$ and let $\tilde y = t \tilde e_2 + (1 - t)\tilde e_1$ be such an element of the edge  that connects $\tilde e_2$ and $\tilde e_1$ that  $\tilde e_2^*(\tilde y) = t$. Then if for a given $\alpha  > 0$ there is $\tilde x \in \Face(\tilde e_2^*)$ such that
\beq \label{eq-hexag-prA}
\|\tilde x  - \alpha\tilde y\| = 1 - \alpha t
\eeq
then $\alpha \le 1$.
\end{propos}
\begin{proof}
Indeed, assume that $\alpha >1$. Let us write $\alpha\tilde y$ in coordinate form $\alpha\tilde y = (a_1, a_2)$, $a_1, a_2 \ge 0$. Remark that  \eqref{eq-hexag-prA} implies that $a_2 = \tilde e_2^*(\alpha\tilde y) = \alpha t \le 1$, so $1 < \alpha = \|\alpha\tilde y\| = a_1 + \frac{1}{2}a_2 = a_1 + \frac{1}{2} \alpha t$, in particular
$$
a_1 > 1 - \frac{1}{2} \alpha t \ge  \frac{1}{2}.
$$
The vector $\tilde x$ is of the form $\tilde x = (b_1, 1)$ with $|b_1| \le \frac{1}{2}$. So $ \tilde x  - \alpha\tilde y = (b_1 - a_1, 1 - \alpha t)$, and $\|\tilde x  - \alpha\tilde y\| \ge |b_1 - a_1| + \frac{1}{2}(1 - \alpha t) = a_1 - b_1 + \frac{1}{2}(1 - \alpha t) \ge a_1 - \frac{1}{2} + \frac{1}{2}(1 - \alpha t) = a_1 - \frac{1}{2} \alpha t > 1 - \alpha t$. \end{proof}

\section{Direct sums of GL-spaces} \label{sec-sums}

Let  $E=(\R^n,\|\cdot\|_E)$ be a normed space, and denote $e_k$, $k=1, \ldots, n$, the elements of the canonical basis: $e_1 = (1, 0, \ldots, 0)$, $e_2 = (0, 1, 0, \ldots , 0)$, etc. The  norm $\|\cdot\|_E$ is called \emph{absolute} if it satisfies the following conditions:
\begin{enumerate}
 \item[(i)] $\|e_k\|_E = 1$,  $k=1, \ldots, n$;
 \item[(ii)] for every $a = (a_1, ... , a_n)$ the vector $|a| := (|a_1|, ... , |a_n|)$ has the same norm as $a$:
 $$
 \|(a_1, ... , a_n)\|_E = \|(|a_1|, ... , |a_n|)\|_E.
 $$
\end{enumerate}
The above properties imply that the \emph{norm is monotone} in the following sense: if $a = (a_1, ... , a_n)$ and $b = (b_1, ... , b_n)$ satisfy $0 \le a_k \le b_k$, $k=1, \ldots, n$, then  $\|a\|_E \le  \|b\|_E$.

The dual $E^*$ to $E$ we identify in the standard way with $(\R^n,\|\cdot\|_{E^*})$, where a functional $b = (b_1, ... , b_n) \in E^*$  acts on $a = (a_1, ... , a_n)  \in E$ by the formula $b(a) = b_1 a_1 + \ldots + b_n a_n$. Remark that the norm $\|\cdot\|_{E^*}$ is also absolute.

Let $X_1, \ldots, X_n$ be normed spaces, and $E=(\R^n,\|\cdot\|_E)$ be a space with absolute norm. The $E$-sum of the spaces $X_k$ (we will denote it $E(X_k)_{1}^n$) is the vector space of all $n$-tuples $x = (x_1, ... , x_n)$, $x_k \in X_k$, $k=1, \ldots, n$, equipped with the norm
$$
\|x\| = (\|x_1\|, \ldots, \|x_n\|)_E.
$$
In order to shorten the notation, for $x = (x_1, ... , x_n) \in E(X_k)_{1}^n$ we denote $N(x) = (\|x_1\|, \ldots, \|x_n\|)$. In this notation, $\|x\| = (N(x))_E$.
The dual space to $E(X_k)_{1}^n$ is $E^*(X_k^*)_{1}^n$, where $f = (f_1, ... , f_n) \in E^*(X_k^*)_{1}^n$  acts on $x = (x_1, ... , x_n)  \in E(X_k)_{1}^n$ by the rule $f(x) = f_1(x_1) + \ldots + f_n(x_n)$.

\begin{definition}\label{def-GL-respecting}
\emph{ A space $E=(\R^n,\|\cdot\|_E)$ with absolute norm is said to be} GL-respecting \emph{(GLR-space for short) if for every collection $X_1, \ldots, X_n$ of  GL-spaces  the corresponding  $E$-sum  $E(X_k)_{1}^n$ is generalized lush}.
\end{definition}

It is known \cite[Theorem 2.11]{TanHuangLiu} that $\ell_1^n$ and $\ell_\infty^n$ are GL-respecting (formally speaking, that theorem deals with infinite sums, but it remains valid for finite sums as well). Our goal in this section is to find out, if there are other examples of GLR-spaces. We start with an evident remark.

\begin{propos}\label{prop-easy}
Every GLR-space is a GL-space.
\end{propos}
\begin{proof}
Just substitute into the definition of GLR-space $X_k = \R$, $k=1, \ldots, n$.
\end{proof}

To tell the truth, when we started to think about a possible description of GLR-spaces, we expected that the converse to Proposition \ref{prop-easy} should be true. Surprisingly, this is not the case. Before we formulate another necessary condition for being a GLR-space, let us make some preliminary observations. From Proposition \ref{prop-easy} and Theorem \ref{theo-main1}  it follows that the unit ball of every GLR-space $E$ is a polyhedron. Then the unit ball of $E^*$  is a polyhedron as well. Denote $\ex(B_{E^*})$ the set of extreme points of $B_{E^*}$. The set $\ex(B_{E^*})$ is finite, the number of elements in $\ex(B_{E^*})$ is the same as the number of faces of  $B_{E}$ and each $d \in \ex(B_{E^*})$ corresponds to a face $\Face(d)$ of  $B_{E}$. Since the norm $\|\cdot\|_{E^*}$ is absolute, $\ex(B_{E^*})$ is mirror-symmetric with respect to each coordinate hyperplane, that is $d \in \ex(B_{E^*})$ \ifff $|d| \in \ex(B_{E^*})$.

\begin{definition}\label{def-GL-monotone}
\emph{Let $E=(\R^n,\|\cdot\|_E)$ be a space  with absolute norm. A face $\Face(d^*) \subset S_E$,   $d^* = (d_1, ... , d_n)  \in \ex(B_{E^*})$ is said to be} monotone plump \emph{if, denoting $D=\{k: d_k \neq 0\}$, for every $a  = (a_1, ... , a_n)  \in S_{E}$  with $a_k \ge 0$ and every $z = (z_1, ... , z_n) \in B_E$ with
\beq \label{eq_z_k>0}
0 \le  z_k \le a_k \textrm{ for } k \in D
\eeq
there is such a $b = (b_1, ... , b_n)  \in S_{E}$ that $d^*(b) = 1$, $\|b - z\| = 1 - d^*(z)$ and $b_k \ge a_k$ for $k \in D$}.  \emph{The space $E$ is said to be} GL-monotone \emph{(GLM-space for short) if for every $d^* \in \ex(B_{E^*})$ the corresponding face $\Face(d^*) \subset S_E$ is monotone plump}.
\end{definition}
\begin{propos}\label{prop-GL-monotone-conseq}
In the above notation let  $d^* = (d_1, ... , d_n)  \in \ex(B_{E^*})$ generate a monotone plump face. Then
\begin{enumerate}
 \item all the coordinates $d_k$ belong to $\{0, 1, -1\}$;
 \item the property formulated in Definition \ref{def-GL-monotone} remains valid for every $z = (z_1, ... , z_n) \in B_E$ that satisfies the following weaker version of \eqref{eq_z_k>0}: $|z_k| \le a_k$  for $k \in D$.
\end{enumerate}
\end{propos}
\begin{proof}
By symmetry, it is sufficient to consider $d^* = (d_1, ... , d_n)  \in \ex(B_{E^*})$ with $d_k \ge 0$. Since $\Face(d^*)$  is an $(n-1)$-dimensional affine subset of the  sphere, it cannot be contained in a kernel of a coordinate functional. So the set
$$
\{x  = (x_1, ... , x_n)  \in \Face(d^*) \colon \exists_{ k \in \{1, \ldots , n\}} \, x_k = 0  \}
$$
is nowhere dense in $\Face(d^*)$. Take an $x  = (x_1, ... , x_n)  \in \Face(d^*)$ with all non-zero coordinates and consider $a  = (a_1, ... , a_n)  \in S_{E}$  with $a_k  = |x_k| > 0$, $k=1, \ldots, n$. Then
$$
1 = d^*(x) \le d^*(a) \le \|a\| =  \|x\| = 1,
$$
which means that $a \in \Face(d^*)$. Fix a $j \in D$. Our goal is to demonstrate that $d_j = 1$. Apply Definition \ref{def-GL-monotone} to such a $z = (z_1, ... , z_n) \in B_E$ that $z_j = 0$, and $z_k = a_k$ for $k \neq j$. We get a $b = (b_1, ... , b_n)  \in S_{E}$ with $d^*(b) = 1$, $\|b - z\| = 1 - d^*(z)$ and $b_k \ge a_k$ for $k \in D$. Remark that
$$
1 = d^*(a) = \sum_{k \in D} d_k a_k \le \sum_{k \in D} d_k b_k = d^*(b) = 1,
$$
which implies that $b_k = a_k$ for $k \in D$, and in particular the j-th coordinate of $b-z$ is equal to $a_j$. Finally,
$$
a_j \ge d_j a_j = \sum_{k \in D} d_k(a_k - z_k) = \sum_{k \in D} d_k(b_k - z_k) = 1 - d^*(z) = \|b - z\| \ge a_j.
$$
Consequently, $d_j = 1$. This completes the proof of the statement (1).

For the statement (2), remark first that for every $u = (u_1, ... , u_n) $ such that $u_k \ge 0$ for  $k \in D$ and $u_k = 0$ for  $k \not\in D$ the following equality holds:
$$
d^*(u) = \|u\|.
$$
Indeed, thanks to (1) we have
$$
\|u\| \ge d^*(u) = \sum_{k \in D} u_k = \sum_{k \in D} u_k \|e_k\| \ge \|u\|.
$$
Now consider $a  = (a_1, ... , a_n)  \in S_{E}$  with $a_k \ge 0$, a $z = (z_1, ... , z_n) \in B_E$ with $|z_k| \le a_k$, $k \in D$. Define $w  = (w_1, ... , w_n)  \in S_{E}$ with $w_k = |z_k|$  for $k \in D$ and $w_k = z_k$  for remaining values of $k$. Find a $b = (b_1, ... , b_n)  \in S_{E}$ that serves in  Definition \ref{def-GL-monotone} for $w$, that is such that $d^*(b) = 1$, $\|b - w\| = 1 - d^*(w)$ and $b_k \ge a_k$ for $k \in D$.  Let us demonstrate that the same $b$ serves for $z$:
$\|b - z\| \le  \|b - w\| + \|w - z\| = 1 - d^*(w) + d^*(w - z)  =  1 - d^*(z) = d^*(b - z) \le \|b - z\|$.
\end{proof}

\begin{propos}\label{prop-zero-one}
Let $E=(\R^n,\|\cdot\|_E)$ be GL-respecting. Then $E$ is GL-monotone.
\end{propos}
\begin{proof}
It is sufficient to consider $d^* = (d_1, ... , d_n)  \in \ex(B_{E^*})$ with $d_k \ge 0$ and demonstrate that  $\Face(d^*) \subset S_E$ is monotone plump. As before, we denote $D=\{k: d_k \neq 0\}$. Substitute into the definition of GLR-space $n$ isometric copies of the hexagonal space $\tilde E$. Denote these copies $X_k$, $k=1, \ldots, n$. Fix $t_k \in [0, 1]$ such that $t_k a_k = z_k$ and for each $k=1, \ldots, n$ take a functional $x_k^* \in  \ex(B_{X_k^*})$ being a copy of the second coordinate functional on  $\tilde E$,   and take an element $y_k \in S_{X_k}$ belonging to the copy of the edge connecting $\tilde e_2$ and $\tilde e_1$ in $\tilde E$  such that $x_k^*(y_k) = t_k$ (like $\tilde y$ from Proposition \ref{propA}). The following \textbf{Property A} that will be used below is a direct consequence of Proposition \ref{propA}: if for a given $\alpha  > 0$ there is $x_k \in \Face(x_k^*)$ such that $\|x_k - \alpha y_k\| = 1 - \alpha t_k$, then $\alpha \le 1$.

By definition of GLR-space, the space $X = E(X_k)_{1}^n$ is generalized lush. Consider $x^*= (d_1 x_1^*, ... , d_nx_n^*) \in \ex(B_{X^*})$. Since $X$ is finite-dimensional,  the corresponding face $\Face(x^*)$ of  $B_{X}$ must be plump. Applying the reformulation of plumpness given in (4) of Lemma \ref{lemma-plump-face-description} to $y = (a_1 y_1, ... , a_n y_n) \in S_X$ with arbitrary $a  = (a_1, ... , a_n)  \in S_{E}$, $a_k \ge 0$, we find an $x = (b_1 x_1, ... , b_n x_n) \in \Face(x^*)$ with $x_k \in  S_{X_k}$, $b_k \ge 0$,  $b = (b_1, ... , b_n)  \in S_{E}$ such that $\|x - y\| = 1 - x^*(y)$. The condition $x \in \Face(x^*)$ implies that
\bea
1 & \ge \sum_{k=1}^n d_k b_k \ge  \sum_{k=1}^n d_k b_k x_k^*(x_k) = x^*(x) = 1.
\eea
This means that $ \sum_{k=1}^n d_k b_k = 1$ and for every $k \in D$ we have $b_k x_k^*(x_k) = b_k$. Using the properties listed above we obtain the following chain of inequalities
\bea
1 &-& x^*(y) = \sum_{k=1}^n d_k x_k^*( b_k x_k -  a_k y_k) \le \sum_{k=1}^n d_k \|b_k x_k -  a_k y_k\| \\
&\le& \left\| (\|b_1 x_1 -  a_1 y_1\|, \ldots, \|b_n x_n -  a_n y_n\| )\right\|_E = \|x - y\| = 1 - x^*(y).
\eea
From this we deduce that for every  $k \in D$ we have $\|b_k x_k -  a_k y_k\| = x_k^*( b_k x_k -  a_k y_k)$. The Property A of hexagonal norm gives us $a_k \le b_k$  for every  $k \in D$. On the other hand,
\bea
1 &-& x^*(y) = \sum_{k=1}^n d_k (x_k^*( b_k x_k) -  x_k^*(a_k y_k))  \\
&\le& \left\| (|x_1^*( b_1 x_1) -  x_1^*(a_1 y_1)|, \ldots, |x_n^*( b_n x_n) -  x_n^*(a_n y_n)| )\right\|_E  \\
&\le& \left\| (\|b_1 x_1 -  a_1 y_1\|, \ldots, \|b_n x_n -  a_n y_n\| )\right\|_E = \|x - y\| = 1 - x^*(y),
\eea
So, $\sum_{k=1}^n d_k (x_k^*( b_k x_k) -  x_k^*(a_k y_k))$ is equal to
$$
\left\| (|x_1^*( b_1 x_1) -  x_1^*(a_1 y_1)|, \ldots, |x_n^*( b_n x_n) -  x_n^*(a_n y_n)| )\right\|_E.
$$
Substituting $x_k^*(x_k) = 1$, $x_k^*(y_k) = t_k$ we obtain
$$
\sum_{k=1}^n d_k (b_k -  t_k a_k) = \left\| (| b_1 -  t_k a_1|, \ldots, |b_n -  t_k a_n| )\right\|_E,
$$
that is $d^*(b - z) = \|b - z\|$.
\end{proof}

\begin{corollary}
 The only two-dimensional GLR-spaces are $\ell_1^2$ and $\ell_\infty^2$.
 \end{corollary}
 \begin{proof}
 Proposition \ref{prop-easy} and Theorem \ref{theo-main2} imply that in dimension 2 there are no other candidates for being GLR-space except for those spaces whose unit ball is either parallelogram or hexagon. Let $X=(\R^2,\|\cdot\|)$ be a GLR-space. Propositions \ref{prop-GL-monotone-conseq} and \ref{prop-zero-one} imply that for every $(x,y)\in \ex(B_{X^*})$   $x,y \in \{-1,0, 1\}$. Let us consider two cases.

\textbf{Case 1.}  $(1,1) \in \ex(B_{X^*})$. Then by symmetry all four points $(\pm 1, \pm1) \in \ex(B_{X^*})$, i.e.  $X^*=\ell_\infty^2$ and consequently $X=\ell_1^2$.

\textbf{Case 2.}  $(1,1) \notin \ex(B_{X^*})$. Then $(1,0), (-1,0), (0,1), (0,-1) \in \ex(B_{X^*})$, i.e.  $X^*=\ell_1^2$ and consequently $X=\ell_\infty^2$.
 \end{proof}

Now we are starting preparations for the inverse to Proposition \ref{prop-zero-one} statement.

\begin{propos}\label{prop-difficult}
Let $E=(\R^n,\|\cdot\|_E)$ be a GL-monotone space. Then for every collection $X_1, \ldots, X_n$ such that each $X_k$ is ultra-GL with respect to a subspace $W_k \subset X_k^*$ the corresponding  $E$-sum $X = E(X_k)_{1}^n$ is ultra-GL with respect to the subspace $W = E^*(W_k)_{1}^n \subset X^*$. In particular, if $X_1, \ldots, X_n$ are ultra-GL then $X$ is ultra-GL as well.
\end{propos}
\begin{proof}
According to Definition \ref{def-U-GL}, for a given $x = (x_1, ... , x_n) \in S_X$ we must find an $x^* = (x_1^*, ... , x_n^*) \in S_{W}$ such that $x \in \Face(x^*)$ and $\Face(x^*)$ is plump. Let us do it. For each $k=1, \ldots, n$, using that $X_k$ is ultra-GL with respect to $W_k$, let us find a $w_k^* \in S_{W_k}$ such that $\Face(w_k^*) \subset S_{X_k}$ is plump in $X_k$ and $x_k \in \|x_k\|\Face(w_k^*)$, that is
\begin{equation} \label{prop-difficult-eq0}
w_k^*(x_k) = \|x_k\|.
\end{equation}
Also, monotone GL-ness of $E$ applied to $N(x) \in S_E$ gives us a $d^* =  (d_1, ... , d_n) \in \ex(B_{E^*})$, $d_k \ge 0$, such that $N(x) \in \Face(d^*) \subset S_E$,
\begin{equation} \label{prop-difficult-eq1}
\sum_{k=1}^n d_k \|x_k\| = 1
\end{equation}
and $\Face(d^*)$ is monotone plump in $E$. Let us demonstrate that $x_k^* := d_k w_k^*$, $k=1, \ldots, n$,  generate the functional $x^* = (x_1^*, ... , x_n^*) \in S_{W}$ we need. Indeed, the conditions \eqref{prop-difficult-eq0} and \eqref{prop-difficult-eq1} imply that $x \in \Face(x^*)$, so it remains to show that $\Face(x^*)$ is plump. As above, denote  $D = \{k: d_k \neq 0\}$. Consider an arbitrary $y = (a_1y_1, ... , a_n y_n) \in S_X$, $y_k \in S_{X_k}$, $a_k \ge 0$. For $z =  (z_1, ... , z_n) =  (a_1w_1^*(y_1), ... , a_nw_n^*(y_n)) \in B_{E}$ there is an element $b = (b_1, ... , b_n) \in \Face(d^*)$ such that $\|b - z\|_E = 1 - d^*(z)$ and $b_k \ge a_k$ for all $k \in D$.

According to (3) of Lemma \ref{lemma-plump-face-description}, for every $k \in D$ such that $b_k \neq 0$ there is a $w_k \in \Face(w_k^*)$ such that $\bigl\| w_k - \frac{a_k y_k}{b_k}\bigr\| =| 1 - a_kw_k^*\left( \frac{y_k}{b_k}\right)|$.  If $k \in D$ and $b_k = a_k = 0$ take  $w_k \in \Face(w_k^*)$ arbitrarily. With such an election, for every  $k \in D$ we have
\begin{equation} \label{equat-kak-nado}
\bigl\| b_k w_k - a_k y_k \bigr\| = |b_k - a_k w_k^*(y_k)|.
\end{equation}
For $k \not\in D$ the election of $w_k^*$ does not affect the value of $x_k^* = d_k w_k^* = 0$. So we can take for $k \not\in D$ as $w_k^*$ a supporting functional at $y_k$ and take $w_k = y_k$. Then
$\bigl\| b_k w_k - a_k y_k \bigr\| = |b_k - a_k| = |b_k - a_k w_k^*(y_k)|$, so \eqref{equat-kak-nado} remains valid for all $k$. Put $\tilde x = (b_1w_1, ... , b_n w_n) $. Then $\tilde x \in \Face(x^*)$ and
\begin{align*} \label{equat-kak-nado2}
1 - x^*(y) &= \sum_{k=1}^n d_k(b_k - a_k w_k^*(y_k)) = 1 - d^*(z) = \|b - z\|_E \\
&= \left\| (|b_1 - a_1 w_1^*(y_1)|, \ldots, |b_n - a_n w_n^*(y_n)|)\right\|_E \\
&=  \left\| (\|b_1w_1 - a_1 y_1\|, \ldots, \|b_1w_1 - a_1 y_1\|) \right\|_E = \|\tilde x - y\|.
\end{align*}
\end{proof}

In order to complete our paper it remains to apply the standard ultraproduct technique. We refer for instance to \cite[Ch. 16]{Kad} for properties of filters and ultrafilters and to classical paper \cite{Heinrich} for  introduction to ultrapowers of Banach spaces. Let us recall the basic definitions.

\begin{definition}
\emph{A family $\mathfrak F$ of subsets of the set $\N$ is called a} filter on $\N$ \emph{if it satisfies the
following axioms:}
\begin{enumerate}
  \item $\mathfrak F$ \emph{is not empty};
  \item $\emptyset \notin \mathfrak F$;
  \item \emph{if} $A, B \in \mathfrak F$, \emph{then} $A \cap B \in \mathfrak F$;
  \item \emph{if} $A \in \mathfrak F$ \emph{and} $A \subset B \subset \N$, \emph{then} $B \in \mathfrak F$.
\end{enumerate}
\end{definition}
\begin{definition}
\emph{Let} $Y$ \emph{be a topological space, and} $\mathfrak F$ \emph{be a filter on} $\N$. \emph{The point} $y$ \emph{in}
$Y$ \emph{is called the} limit of the sequence of $y_n \in Y$, $n = 1, 2, \ldots$ with respect to the filter $\mathfrak F$ \emph{(denoted}
$y = \lim\limits_{\mathfrak F} y_n$\emph{), if for any neighborhood} $U$  \emph{of} $y$  \emph{there exists an element} $A \in \mathfrak F$ \emph{such that} $y_n \in U$ for all $n \in A$.
\end{definition}
\begin{definition}
\emph{An} ultrafilter on $\N$ \emph{is a filter on $\N$ that is maximal with respect to inclusion. An ultrafilter is nontrivial, if all its elements are infinite.}
\end{definition}
Observe, that the the existence of nontrivial ultrafilters is guaranteed by Zorn's lemma and the limit with respect to the ultrafilter exists for any bounded sequence of reals.

Let $\mathfrak U$ be a nontrivial ultrafilter on $\N$, $X$ be a Banach space. The ultrapower  $X^{\mathfrak U}$ is the quotient space of $\ell_\infty(X)$ over the subspace of those $x = (x_n) \in \ell_\infty(X)$ for which $\lim\limits_{\mathfrak{U}} \|x_n\| = 0$.  The ultrapower  $X^{\mathfrak U}$ we will view as the space of bounded sequences $x = (x_n)$ with $x_n \in X$, equipped with the norm $\|x\|= \lim\limits_{\mathfrak{U}} \|x_n\|$ under the convention that $x = (x_n)$ and $y = (y_n)$ are considered to be the same element of the ultrapower if $\lim\limits_{\mathfrak{U}} \|x_n - y_n\| = 0$. The ultrapower $(X^*)^{\mathfrak U}$ can be identified with  a subspace of  $(X^{\mathfrak U})^*$ as follows:  every $x^* = (x_n^*)$, $x_n^* \in X^*$, $\sup_n \|x_n^*\| <\infty$ is a  linear functional on $X^{\mathfrak U}$  that acts on every $x = (x_n)$ by the rule $x^*(x)= \lim_{\mathfrak U}x_n^*(x_n)$. For a functional $x^* = (x_n^*)$ its norm is the same as its norm as an element of the ultrapower $(X^*)^{\mathfrak U}$:  $\|x^*\|= \lim\limits_{\mathfrak{U}} \|x_n^*\|$.

The theorem below is modeled after a similar result about narrow operators \cite[Lemma 2.6]{BKSW2005}. An analogous result for lush spaces was demonstrated in \cite[Corollary 4.4]{BKMM2009}. The implication $\mathbf{(1)\Rightarrow(2)}$ is almost contained in demonstration of \cite[Proposition 2.2]{Hardtke}.

\begin{theo}\label{theo-main3}
Let $X$ be a Banach space, and $\mathfrak U$ be a nontrivial ultrafilter on $\N$. Then the following assertions are equivalent:
\begin{enumerate}
\item  $X$ is a GL-space,
\item    $X^{\mathfrak U}$ is ultra-GL with respect to the subspace $W = (X^*)^{\mathfrak U}$.
\end{enumerate}
\end{theo}
\begin{proof}
$\mathbf{(1)\Rightarrow(2)}$. Our goal is to demonstrate that for every  $x  = (x_n)\in S_{X^{\mathfrak U}}$, $x_n \in X$,  there exists an $x^* = (x_n^*) \in S_{W}$ such that $x \in \Face(x^*)$ and $\Face(x^*)$ is plump. Let us do this. Remark, that $\|x\|= \lim\limits_{\mathfrak{U}} \|x_n\| = 1$, so $\lim\limits_{\mathfrak{U}} \left\| \frac{x_n}{\|x_n\|} - x_n \right\| = 0$, so $(x_n) = \left( \frac{x_n}{\|x_n\|} \right)$ as elements of the ultrapower. Substituting if necessary $x_n$ by $ \frac{x_n}{\|x_n\|}$ we may assume that $x_n \in S_X$, $n = 1, 2, \ldots$ Since by (1) $X$ is a GL-space, for each $n \in \N$ there is $x_n^* \in S_{X^*}$ such that $x_n \in \Slice(x_n^*,\frac{1}{n}) = S_n$ and \eqref{eq*} holds true for every $y\in S_X$. Denote  $x^* = (x_n^*)$. By our construction,
$$
\|x^*\|= \lim\limits_{\mathfrak{U}} \|x_n^*\| = 1 = \lim\limits_{\mathfrak{U}} x_n^*(x_n) = x^*(x).
$$
This means that  $x^* \in S_{W}$ and $x \in \Face(x^*)$.
It remains to show that the corresponding face $\Face(x^*)$ is plump.
 Let  $y  = (y_n)\in S_{X^{\mathfrak U}}$, $y_n \in S_X$. For every $n \in \N$, using \eqref{eq*}, pick $u_n^1, u_n^2 \in S_n$ in such a way that
$$
\|y_n - u_n^1\| + \|y_n + u_n^2\| \leq 2 + \frac{1}{n}.
$$
Denote $u_i  = (u_n^i) \in S_{X^{\mathfrak U}}$, $i = 1,2$. We have $u_i \in \Face(x^*)$ and $ \|y - u_1\| + \|y + u_2\| \leq 2$ which by definition \ref{defin-plump} means that indeed  $\Face(x^*)$ is plump.

$\mathbf{(2)\Rightarrow(1)}$. This time let $x \in S_X$ be an arbitrary element. Assume to the contrary that there is an $\eps > 0$ such that for every $x^* \in S_{X^*}$ with $x^*(x) > 1 - \eps$ there is a $y \in S_X$ with
$$
\dist(y, \Slice(x^*, \eps)) + \dist(- y, \Slice(x^*, \eps)) \ge 2 + \eps.
$$
Take $\tilde x  = (x, x, x, \ldots) \in S_{X^{\mathfrak U}}$ and let us show that for every $x^* = (x_n^*) \in S_{W}$ such that $\tilde x \in \Face(x^*)$ the $\Face(x^*)$ is not plump. Indeed, denote $A = \{n \in \N \colon x_n^*(x) > 1 - \eps\}$. The condition $1 = x^*(\tilde x) = \lim\limits_{\mathfrak{U}} x_n^*(x)$ imply that $A \in \mathfrak U$. Using our assumption, for every $n \in A$ take  a $y_n \in S_X$ with
\begin{equation} \label{equa-ultr-contrad}
\dist(y_n, \Slice(x_n^*, \eps)) + \dist(- y_n, \Slice(x_n^*, \eps)) \ge 2 + \eps.
\end{equation}
Denote $\tilde y = (y_n) \in S_{X^{\mathfrak U}}$. If the face $\Face(x^*)$ were plump,  there would be $u_i  = (u_n^i) \in \Face(x^*)$, $i = 1,2$, such that $ \|\tilde y - u_1\| + \|\tilde y + u_2\| \leq 2$ . This contradicts \eqref{equa-ultr-contrad}.
\end{proof}

Now, finally, the main result of the section.

\begin{theo}\label{theo-main4}
A space $E=(\R^n,\|\cdot\|_E)$ with absolute norm is GL-respecting if and only if it is GL-monotone.
\end{theo}
\begin{proof}
The ``only if'' part is demonstrated in Proposition \ref{prop-zero-one}, so it remains to demonstrate the ``if'' part. Let $E$ be GL-monotone, and let $X_1, \ldots, X_n$ be a collection of  GL-spaces. Due to the previous theorem, for a fixed nontrivial ultrafilter $\mathfrak U$ on $\N$ all $X_k^{\mathfrak U}$, $k = 1, \ldots, n$, are ultra-GL with respect to the corresponding subspaces $(X_k^*)^{\mathfrak U}$.  By Proposition \ref{prop-difficult} the  $E$-sum $E(X_k^{\mathfrak U})_{1}^n$ is ultra-GL with respect to the subspace $E^*\left((X_k^*)^{\mathfrak U}\right)_{1}^n$. Using the natural isometry between $E(X_k^{\mathfrak U})_{1}^n$ and $\left(E(X_k)_{1}^n\right)^{\mathfrak U}$ we deduce that  $\left(E(X_k)_{1}^n\right)^{\mathfrak U}$  is ultra-GL with respect to the subspace $\left(E^*(X_k^*)_{1}^n\right)^{\mathfrak U}$. Another application of the previous theorem gives us the desired  generalized lushness of  $E(X_k)_{1}^n$.
\end{proof}

\bibliographystyle{amsplain}

\end{document}